\documentclass[12pt]{amsart}
\usepackage{amscd}
\usepackage[arrow,matrix]{xy}
\usepackage{graphicx}
\usepackage{amsmath}
\usepackage{color}

\usepackage{amsmath, latexsym, amssymb}
\input xypic
\numberwithin{equation}{section}
\theoremstyle{plain}
\newtheorem{lemma}{Lemma}[section]
\newtheorem{proposition}[lemma]{Proposition}

\newtheorem{corollary}[lemma]{Corollary}

\theoremstyle{definition}
\newtheorem{definition}[lemma]{Definition}
\newtheorem{remark}[lemma]{Remark}
\newtheorem{example}[lemma]{Example}

\begin{document}
\newcommand{\R}{{\mathbb R}}
\newcommand{\C}{{\mathbb C}}
\newcommand{\F}{{\mathbb F}}
\renewcommand{\O}{{\mathbb O}}
\newcommand{\Z}{{\mathbb Z}} 
\newcommand{\N}{{\mathbb N}}
\newcommand{\Q}{{\mathbb Q}}
\renewcommand{\H}{{\mathbb H}}

\newcommand{\Aa}{{\mathcal A}}
\newcommand{\Bb}{{\mathcal B}}
\newcommand{\Cc}{{\mathcal C}}    
\newcommand{\Dd}{{\mathcal D}}
\newcommand{\Ee}{{\mathcal E}}
\newcommand{\Ff}{{\mathcal F}}
\newcommand{\Gg}{{\mathcal G}}    
\newcommand{\Hh}{{\mathcal H}}
\newcommand{\Kk}{{\mathcal K}}
\newcommand{\Ii}{{\mathcal I}}
\newcommand{\Jj}{{\mathcal J}}
\newcommand{\Ll}{{\mathcal L}}    
\newcommand{\Mm}{{\mathcal M}}    
\newcommand{\Nn}{{\mathcal N}}
\newcommand{\Oo}{{\mathcal O}}
\newcommand{\Pp}{{\mathcal P}}
\newcommand{\Qq}{{\mathcal Q}}
\newcommand{\Rr}{{\mathcal R}}
\newcommand{\Ss}{{\mathcal S}}
\newcommand{\Tt}{{\mathcal T}}
\newcommand{\Uu}{{\mathcal U}}
\newcommand{\Vv}{{\mathcal V}}
\newcommand{\Ww}{{\mathcal W}}
\newcommand{\Xx}{{\mathcal X}}
\newcommand{\Yy}{{\mathcal Y}}
\newcommand{\Zz}{{\mathcal Z}}

\newcommand{\zt}{{\tilde z}}
\newcommand{\xt}{{\tilde x}}
\newcommand{\Ht}{\widetilde{H}}
\newcommand{\ut}{{\tilde u}}
\newcommand{\Mt}{{\widetilde M}}
\newcommand{\Llt}{{\widetilde{\mathcal L}}}
\newcommand{\yt}{{\tilde y}}
\newcommand{\vt}{{\tilde v}}
\newcommand{\Ppt}{{\widetilde{\mathcal P}}}
\newcommand{\bp }{{\bar \partial}} 

\newcommand{\ad}{{\rm ad}}
\newcommand{\Om}{{\Omega}}
\newcommand{\om}{{\omega}}
\newcommand{\eps}{{\varepsilon}}
\newcommand{\Di}{{\rm Diff}}
\newcommand{\Aff}{{\rm Aff}}

\renewcommand{\a}{{\mathfrak a}}
\renewcommand{\b}{{\mathfrak b}}
\newcommand{\e}{{\mathfrak e}}
\renewcommand{\k}{{\mathfrak k}}
\newcommand{\pg}{{\mathfrak p}}
\newcommand{\g}{{\mathfrak g}}
\newcommand{\gl}{{\mathfrak gl}}
\newcommand{\h}{{\mathfrak h}}
\renewcommand{\l}{{\mathfrak l}}
\newcommand{\n}{{\mathfrak n}}
\newcommand{\s}{{\mathfrak s}}
\renewcommand{\o}{{\mathfrak o}}
\newcommand{\so}{{\mathfrak so}}
\renewcommand{\u}{{\mathfrak u}}
\newcommand{\su}{{\mathfrak su}}
\newcommand{\ssl}{{\mathfrak sl}}
\newcommand{\ssp}{{\mathfrak sp}}
\renewcommand{\t}{{\mathfrak t }}
\newcommand{\Cinf}{C^{\infty}}
\newcommand{\la}{\langle}
\newcommand{\ra}{\rangle}
\newcommand{\half}{\scriptstyle\frac{1}{2}}
\newcommand{\p}{{\partial}}
\newcommand{\notsub}{\not\subset}
\newcommand{\iI}{{I}}               
\newcommand{\bI}{{\partial I}}      
\newcommand{\LRA}{\Longrightarrow}
\newcommand{\LLA}{\Longleftarrow}
\newcommand{\lra}{\longrightarrow}
\newcommand{\LLR}{\Longleftrightarrow}
\newcommand{\lla}{\longleftarrow}
\newcommand{\INTO}{\hookrightarrow}

\newcommand{\QED}{\hfill$\Box$\medskip}
\newcommand{\UuU}{\Upsilon _{\delta}(H_0) \times \Uu _{\delta} (J_0)}
\newcommand{\bm}{\boldmath}

\title[Constructing elusive functions]{\large Constructing elusive functions  with  help  of  evaluation mappings}
\author{H\^ong V\^an L\^e } 
\thanks{Supported by  RVO: 67985840}
\address{Institute of Mathematics of ASCR, Zitna 25, 11567 Praha,\\ email: hvle@math.cas.cz}

\medskip

\abstract{We  develop a method  to construct elusive functions  using  techniques  of commutative algebra and algebraic geometry.  The key notions   of this method  are  elusive  subsets and evaluation mappings.  We  also develop the effective elimination theory combined  with algebraic number field theory in order  to construct concrete  points   outside the image  of a polynomial mapping. Using the developed methods, for $\F = \C \text{ or } \R$,  we construct examples of    $(s,r)$-elusive  functions whose monomial coefficients are algebraic numbers,  which   give polynomials  with algebraic number coefficients of large  circuit size.} 
\endabstract
\maketitle

{\it AMSC: 03D15, 68Q17, 13P25}

\tableofcontents

\section{Introduction}

In computational  algebraic complexity theory we investigate   different complexity classes
of  sequences $(f_n)$ of  polynomials over a field $\F$. We also search 
  lower or upper bounds of complexities on a given  polynomial.   

 Two most important complexities of a multivariate polynomial $f$ are   the  circuit complexity $L(f)$  and  the formula size $L_e(f)$.
These complexities  measure  the  minimal size of  certain  arithmetic circuits computing  $f$. Arithmetic circuits are the standard computational model for computing polynomials.    {\it An arithmetic circuit}, as defined, e.g., in \cite[\S 1.1]{Raz2009}, is a finite  directed acyclic graph whose  nodes are divided into four types:    nodes of in-degree 0 (input gates) labeled with  an input variable or  the field element 1, nodes labelled with +  (sum gates),  node labeled with $\times$ (product gates), and  nodes of out-degree 0 (output gates) giving the result of the computation. Every edge $(u,v)$ in the graph is labeled with a field element $\alpha$. It computes  the product of $\alpha$ with the polynomial computed by $u$. A product gate (resp. a sum gate) computes the product (resp. the sum) of  polynomials computed by the
  edges  that reach it.  We say that a polynomial  $f \in \F [X_1, \cdots, X_n]$ is computed by a circuit if it is computed by one of the circuit output gates. If a  circuit has $m$ output gates, then it computes a $m$-tuple of polynomials
$f^i \in \F [X_1, \cdots, X_n], \, i \in[1,m]$. In what follows we consider only ordered $m$-tuples of polynomials
resulting from  a  numeration  of the output gates  of an arithmetic circuit; so  an $m$-tuple is understood as an ordered
$m$-tuple. Further, assuming in this note  that $\F$ is a field of characteristic 0, we also identify an $m$-tuple  of polynomials in $n$ variables  with a polynomial mapping  from $\F^n$ to $\F^m$.
Let us denote by $Pol^r (\F^n , \F^m )$  the space of all polynomial mapping of degree at most $r$ from $\F^n$ to $\F^m$  and set $Pol (\F^n, \F^m): = \cup_{r = 0} ^\infty Pol ^r (\F ^n, \F ^m)$.  

 We define {\it the size  of a circuit} as the number of its  edges, and {\it the  circuit complexity} $L(f)$ of a polynomial  mapping $f$ 
to be  the minimum size of an arithmetic circuit computing  $f$ \cite{Raz2009}.  {\it The formula size $L_e(f)$ of a polynomial mapping $f$}   is  defined as the minimum size of an arithmetic circuit computing $f$,  which
is a directed tree, i.e.,  all vertices  have out-degree at most 1.

The  formula size and the circuit complexity of polynomial mappings do not have  clear geometric or algebraic structure. In \cite{Valiant1979} Valiant   suggested to  ``approximate" the formula size  of a polynomial by the determinantal complexity, observing that on the one hand, the determinantal complexity is a lower bound for  the formula size, and on the other  hand, the determinantal complexity has a clear algebraic and geometric interpretation. Geometric and algebraic properties of the determinantal complexity  of a polynomial  have been employed by Mignon-Ressayre \cite{M-R2004}  and   by Mulmuley-Sohoni
\cite{M-S2001} to study  lower bounds on the determinantal complexity, and to attack the problem $VP$ versus $VNP$. 
  
In \cite{Raz2009} Raz proposed a geometric approach to obtain a  lower bound on the   circuit complexity  of a  polynomial by
introducing a polynomial mapping associated with a universal graph of a given arithmetic circuit.  Using his method  Raz  has constructed   explicit  polynomials  whose  constant  depth circuit  size is large \cite[Lemma 4.1]{Raz2009}, see also
Remark \ref{r2}.

Raz's method  of constructing  elusive  functions is  combinatorial, and it is not  clear how  to  apply his method to find  other examples  of elusive  functions.   In this paper  we develop  an algebraic-geometric  method    for  construction of  elusive functions.
The key notion  of this method   are elusive subsets and evaluation mappings.

The structure of our paper is as follows.    In section  2 we recall the notion of a $(s,r)$-elusive function introduced by Raz in \cite{Raz2009}.  To study $(s,r)$-elusive functions we introduce the notion  of a $(s,r)$-elusive  subset (Definition \ref{pelu}) and we characterize  polynomial mappings whose image contains an $(s,r)$-elusive subset consisting  of $k$ points (Corollary \ref{rske}). This construction leads to the notions of a $(s,r,k)$-elusive function and of  a strong $(s,r)$-elusive  function (Definitions \ref{pelu}, \ref{strel}).  We compare these notions, using an interpolation formula for polynomial mappings (Proposition \ref{Lemma 4.4}, Remark \ref{elstr}). In section 3 we  develop the method invented by Kumar-Lokam-Patankar-Sarma \cite{KLPS2010}  that uses  the effective elimination theory   combined with  algebraic number field theory in order to find  concrete points  $b$  which lie outside  the image of a polynomial mapping $g$, if $g$ is defined over $\Q$, such that the coordinates of $b$ are algebraic numbers
(Proposition \ref{glok}).  Note that our method is close  to    the Strassen-Schnorr-Heintz-Sieveking  method  of constructing polynomials  with algebraic  coefficients which are hard to compute, but our method and their method  yield  different  polynomials which are hard to compute in  different  complexity classes (Remark \ref{polq}.1).
In section 4 we  
construct  
examples of $(s,r)$-elusive functions (Proposition \ref{pola}).  Using this, for $\F =\R$ or $\F =\C$,  we construct explicit  examples of  sequences   of polynomials $f_n:\F^{2n} \to \F^n$ of  degree $5r+1$ whose coefficients are algebraic numbers such that any depth $r$ arithmetic circuit  for  $f_n$  is of size  greater than $ n^2/50r^2$ (Proposition \ref{c45}).  We compare  our results with  previously obtained   results (Remark \ref{r2}). 
We  also  construct  $(s,r)$-elusive  functions whose monomial coefficients  are  algebraic numbers, which give    polynomials  of large  circuit size (Proposition \ref{super}, Corollary \ref{cor:circs}).
  
Finally we note that  our results in  effective elimination theory are applicable for similar complexities of the same nature, e.g.  the determinantal complexity, the rank of tensors and the rigidity of matrices.

\section{Elusive functions  and  associated polynomial mappings}
In this section we recall the notion of a $(s,r)$-elusive function introduced by Raz in \cite{Raz2009} for
constructing   sequences of multivariate polynomials  of high  circuit complexity (Definition \ref{elus}).  To study $(s,r)$-elusive functions we introduce the notion  of a $(s,r)$-elusive  subset (Definition \ref{pelu}) and we find  a condition for a polynomial mapping whose image contains a $(s,r)$-elusive subset  (Corollary \ref{rske}). We also introduce the notion of a $(s,r,k)$-elusive function (Definition \ref{pelu}) and the notion 
of a strongly $(s,r)$-elusive  functions (Definition \ref{strel}). 
We compare  $(s,r)$-elusive functions  with $(s,r,k)$-elusive functions and strongly $(s,r)$-elusive functions, using an interpolation formula for polynomial mappings over $\F$ and  an evaluation mapping (Proposition \ref{Lemma 4.4}, Remark \ref{elstr}).  

\begin{definition}[\cite{Raz2009}, p. 2]\label{elus} A polynomial mapping $f: \F ^ n \to \F  ^ m$ is  called {\it $(s, r)$-elusive},  if for every polynomial mapping $\Gamma : \F ^ s \to \F ^ m$ of degree $r$, we have $f(\F ^ n) \not\subset
\Gamma(\F ^ s)$.
\end{definition}

Using the existence of elusive  functions Raz has constructed  polynomials of large circuit size \cite[\S  3.4]{Raz2009}. Raz's construction of elusive functions is based on a certain combinatoric property
of the coefficients of a special polynomial mapping \cite[Lemma 4.1]{Raz2009}.   Our approach to elusive functions is based on the concept of an $(s,r)$-elusive subset.

\begin{definition} \label{pelu} A  $k$-tuple $S_k$ of $k$  points in $\F ^m$  is called {\it $(s,r)$-elusive},  if for every polynomial mapping $\Gamma : \F ^ s \to \F ^ m$ of degree $r$, we have $S_k \not\subset\Gamma(\F ^ s)$.
A polynomial  mapping $f: \F^n \to \F^m$ is called {\it  $(s,r, k)$-elusive}, if  there is a $k$-tuple  of points in the image $f(\F^n)$ which is $(s,r)$-elusive. 
\end{definition}

Clearly any  $(s,r, k)$-elusive  function is $(s,r)$-elusive.  

\begin{example}\label{mom}(cf. \cite{Raz2009}) A polynomial mapping $f: \F ^ n \to \F ^ m$ is $( m-1,1)$-elusive,
if and only if  the image $f(\F^n)$ does not belong to any hyperplane in the affine space $\F ^ m$. Equivalently, a $(m-1, 1)$-elusive polynomial is $(m-1, 1, m +1)$-elusive.
For example,  the moment curve $f : \C \to \C ^m, \, t \mapsto ( t, t ^2, \cdots , t ^m)$ is $( m-1,1)$-elusive, since the image of the moment curve
contains $m+1$ points $b_0 : = f(0) = 0, \cdots,  b_i : = f(a_i)\in \C^m, \, 1\le i \le m,$ satisfying the following condition. The  values $a_i\in \F^n$ are chosen to be distinct  such that $b_1, \cdots, b_m$ are  
 linear independent vectors in $\C^n$. Clearly the $(m+1)$-tuple $(0, b_1, \cdots, b_m)$ is   $(m-1, 1)$-elusive, which implies that $f$ is $(m-1, 1, m+1)$-elusive, see Corollary \ref{colin} for a detailed explanation.
\end{example}

To treat   $(s,r)$-elusive $k$-tuples  we consider the following evaluation map
\begin{eqnarray}
 Ev^k_{r,s, m}: Pol ^r (\F ^{s}, \F ^m )\times  (\F ^ s)^k  \to (\F ^m ) ^k,\label{evkrsm}\\
 (f_1, \cdots , f_m) ( a_1, \cdots, a_{k}) \mapsto (  f_1 ( a_1), \cdots,
f_m ( a_{k})),\nonumber
\end{eqnarray}
where $f_j\in Pol ^r(\F ^s)$ for  $ 1\le j \le m$ and $ a_i  \in \F ^s$ for $1\le i \le k$.

We identify  a k-tuple $S_k = (b_1, \cdots, b_k)$, $b _i \in \F^m$, with the point $\overline{S_k} \in (\F ^m)^k$ whose coordinate $\overline{S_k}^{i,j}$, $1\le i \le k, \, 1\le j \le m,$ is  equal to the $i$-th coordinate $b_j ^i$ of $b_j \in \F^m$. 

\begin{lemma}\label{ev} A $k$-tuple $S_k\subset \F^m$ is $(s,r)$-elusive, if and only if $\overline {S_k}$ does not belong to the image of $Ev ^k _{r, s,m}$.
\end{lemma}

\begin{proof}  Assume that $\overline{S_k}$ belongs to the image of  $Ev ^k _{ s,r,m}$. Then there are a polynomial mapping $f \in Pol^r (\F ^s, \F ^{m})$ and   a point $a \in \F ^{sk} $ such that
\begin{equation}
Ev ^k _{r,s,m} (f, a) =  \overline{S_k}.\label{eval1}
\end{equation}
We write $S_k= (b_1, \cdots, b_{k}), \, b_i\in \F^m$, and  $ a = (a_1, \cdots, a_{k}), \,  a_i \in \F^s$.  The equation (\ref{eval1}) implies   
\begin{equation}
f (a_i) = b_i.\label{eval2}
\end{equation}
Thus  $S_k\subset f (\F ^s)$. This proves the ``only if" assertion of Lemma \ref{ev}.

Conversely, assume that $S_k \subset f (\F^s)$ for  some $f \in Pol^r (\F ^s, \F ^m)$. Then there are points $a_i \in \F^s$, $i = \overline{1,k}$, such that  (\ref{eval2}) holds for all $i$. Since (\ref{eval2}) is  equivalent to
(\ref{eval1}), it follows  that $\overline{S_k}$  belongs to the image of $Ev ^k_{r, s,m}$. This  completes the proof of Lemma \ref{ev}.
\end{proof}

\begin{corollary}\label{rske} A polynomial map $f:\F ^n \to \F ^m$ is $(s,r, k)$-elusive, if and only if  the subset $\hat f  ^k : = f (\F^n) \times \cdots _{k\, times} \times f (\F^n) \subset  \F^{mk}$ does not belong to
the image  of   the evaluation mapping $Ev ^k_{s,r,m}$. 
\end{corollary}

Now we are going to find a sufficient condition for  a polynomial mapping $f: \F ^n \to \F  ^m$  to be 
$(s,r, k)$-elusive using an interpolation formula for a polynomial mapping. 

Interpolation  of a function  in many variables by a polynomial mapping  has been investigated  for a long time, but there are many interesting and unsolved questions   \cite{GS2000}. One of the main differences between interpolation of a function in one variable and interpolation of
a function in many variables is that in the former case  an interpolable set, i.e., the set at which the value of an interpolating polynomial function (resp. a polynomial mapping) must coincide with the value of a given  interpolable function,  can be arbitrary, but  in the later case cannot be arbitrary.
The interpolation formula  given below is likely unknown, though possibly,  there are some similar formulas.  Our interpolable set is  a lattice  in  a simplex in $\F ^n$.

Note that a monomial $X_1^{i_1} \cdots X_s ^{i_s}  \in Pol ^r( \F  ^s)$ can be identified with an ordered $s$-tuple $(i_1, \cdots, i_s)$ of non-negative integers $i_j$, $1\le j \le s$, such that $i_1 + \cdots + i_s  \le r$.  The following formula is well-known 

\begin{equation}
\dim Pol ^r (\F ^s) = \binom{s+r}{s}.\label{combi1}
\end{equation} 

By (\ref{combi1}) there exists a  1-1  mapping 
 $H^r_s$ from the set $Mon_s^r$ of all monomials $X_1 ^{i_1}\dots X_s^{i_s}\in Pol ^r (\F ^s)$ to the set $S_{s,r}$ of  $\binom{s+r}{r}$ points $(i_1, \cdots,  i_s) \in \F ^s$.   (The  mapping  $H^r_s$  induces a linear isomorphism $H^r_{s, m}: Pol ^r (\F^s, \F^m) \to (\F^s)^m,\,  k = m\binom{s+r}{r}$.) 
For a  set $S_{s,r,m}$ of  $\binom{s+r}{r}$ points
in $\F ^ m$ we enumerate the points in $S_{s,r,m}$ by $b_{i_1, \cdots, i_s}$,  where $i_s \in \N$ and $\sum_s i_s \le r$.

Now we are ready to prove

\begin{proposition}\label{Lemma 4.4} Given a tuple  $S_{s,r,m}$ of $\mathrm{\binom{s+r}{r}}$  points  $b_{i_1, \cdots , i_s}$ in $\F ^m$, $i_j \in \N$ and $\sum_{j=1}^s i_j \le r$, there exists
an algorithmically constructed  polynomial  mapping  $ f_{S_{s,r,m}} : \F ^s \to \F ^m$ of degree $r$ 
such that
\begin{equation}
f_{S_{s,r,m}} (i_1, \cdots, i_s) = b_{i_1, \cdots , i_s}, \label{interp1a}
\end{equation}
for all $(i_1, \cdots, i_s) \in \N ^s\subset \F^s$  satisfying $\sum_{j=1} ^s i_j \le r$.
\end{proposition}

\begin{proof} Let $f^i$  (resp. $b ^i$) denote the i-th  coordinate of a polynomial mapping $f: \F^s \to \F^m$ (resp.  of a point $b \in \F^m)$, i.e., $f = (f^1, \cdots, f^m)$.  Note that (\ref{interp1a}) is equivalent  to the following system of  equations
\begin{equation}
f_{S_{s,r,m}} ^i  (i_1, \cdots, i_s) = b_{i_1, \cdots , i_s} ^i ,\text { for } i \in [1,m]\label{interp1i}
\end{equation}
and for all $(i_1, \cdots, i_s) \in \N ^s\subset \F^s$  satisfying $\sum_s i_s \le r$.
Since the system (\ref{interp1i})  consists of independent  subsystems each of which corresponds to  an upper index $i \in [1,m ]$, it  suffices  to prove  Proposition
\ref{Lemma 4.4}  for the case $m =1$.

 We construct
$f_{S_{s,r,1}}$ by induction  on $s$. Note that the case $s =1$ is well-known.  Given   an $(r+1)$-tuple $( b_0, \cdots, b_r)$ of elements  $b_i \in \F $, there is a polynomial   $f_{S_{1,r,1}} \in \F[X]$   taking values  in $( b_0, \cdots, b_r)$. 
 The Newton  interpolation formula  defines $f_{S_{s,r,1}}$ by the following formula  
\begin{equation}f_{S_{1,r,1}} (X) := \lambda_0 + \lambda_1 X + \lambda_2 X (X  - 1) + \cdots + \lambda_r X( X-1) \cdots (X-r),
\label{4.1.1}
\end{equation}
where the coefficients $\lambda_k\in \F$ are defined inductively on $k$ by solving the system of the following  linear equations  with coefficients in $\N$
$$\lambda_0 = b _0,$$
$$ \lambda_0+ \lambda_ 1 = b_1, $$ 
$$ \cdots $$
\begin{equation}
 \lambda_0 + \lambda_1\cdot k    + \cdots + \lambda_k \cdot k !  = b_k , \label{4.1.1.a}
 \end{equation}
etc.
\medskip

Next, let us assume that $s_0 \ge 2$  and Proposition \ref{Lemma 4.4} is valid for $s\le s_0-1$. Now we show how to construct  the required polynomial $f_{S_{s_0, r,1}}$.
Recall that $ f_{S_{s_0,r,1}}:\F ^ {s_0} \to \F $    is    required to satisfy the following
equation
\begin{eqnarray}
f_{S_{s_0,r,1}} (i_1, i_2, \cdots,  i_{s_0}) = b_{i_1,\cdots i_{s_0}}\in S_{s_0, r,1}\subset \F\label{tsr}\\
 \text { for all } (i_1, \cdots, i_{s_0}) \text { such that } X_1^{i_1} \cdots X_s^{ i_{s_0}} \in Mon _{s_0}^ r.
\nonumber
\end{eqnarray}
We  set
\begin{eqnarray}
f_{S_{s_0,r, 1}}(X_1, \cdots, X_{s_0}) := P ^{  r} ( X_{1}, \cdots, X_{s_0-1}) + X_{s_0} P ^{r-1} ( X_{1}, \cdots , X_{s_0-1})  + \cdots \nonumber\\
+ X_{s_0} ( X_{s_0} -1) \cdots ( X_{s_0} - r+ 1) P ^{0} ( X_{1}, \cdots,  X_{s_0-1}).\label{tsr1}
\end{eqnarray}

To  determine the polynomials $ P^{k} (X_1, \cdots, X_{s_0-1})$ entered in (\ref{tsr1})  for $ 0\le k \le r$  we exploit the following canonical injective map
\begin{equation}
Mon_{s-1}^{ r} \to Mon_s ^r, \, X_1 ^{i_1} \cdots X_{s-1} ^{i_{s-1}} \mapsto X_1 ^{i_1} \cdots X_{s-1} ^{i_{s-1}},\label{incl1}
\end{equation}
as well as  the following canonical inclusions
\begin{equation}
 Mon_{s-1}^ r \supset Mon_{s-1}^{ r-1} \supset Mon_{s-1}^{ r-2} \supset Mon_{s-1}^{ r-3} \supset \cdots. \label{incl2} 
\end{equation}
Using (\ref{incl1}) and (\ref{incl2}) we denote  the restriction of $H^r_s
$ to $Mon_{s_0-1}^ {r-k}$ by
$H_{s_0-1}^{ r-k}$. The image $H_{s_0 -1} ^{r-k} (Mon _{s_0-1} ^{r-k})$ is a set $S_{s_0-1, r-k}$ of $\binom{s_0-1 +r-k}{r-k}$ elements in $\F ^{s_0-1} \subset \F^{s_0}$.
Clearly,  for $0\le k \le r$
$$S_{s_0-1, r -k} = \{  (i_0, \cdots, i_{s_0-1}, k)| \, i_j \in \N  \text { and  }\sum_{j=1} ^ {s-)} i_j  \le r-k\} \subset S_{s_0, r}.$$ 

Next we decompose 
$$S_{s_0, r,1} :  = \{ b _{i_0, i_1, \cdots, i_{s_0}}| \, i_j \in \N  \text { and  }\sum_{j=1} ^ {s_0} i_j  \le r\}\subset \F$$
 as a union of its disjoint subsets
\begin{equation}
S_{s_0, r,1} = S_{s_0-1, r,1} \cup S_{s_0-1, r-1, 1} \cup \cdots \cup S_{s_0-1, 0,1}, \nonumber 
\end{equation}
where for $0\le k \le r$
$$S_{s_0-1, r -k, 1} := \{  b_{i_0, \cdots, i_{s_0-1}, k}| \, i_j \in \N  \text { and  }\sum_{j=1} ^ s i_j  \le r-k\} \subset S_{s_0, r, 1}.$$ 

Substituting $X_{s_0} = 0$ into (\ref{tsr1}), taking into account (\ref{tsr}), we observe  that  the polynomial $P ^ r (X_1, \cdots , X_{s_0-1})$   
satisfies the following  equation
\begin{eqnarray}
P^ r ( i_1, \cdots, i_{s_0-1}) = b_{i_0, \cdots, i_{s_0-1}, 0}\in S_{s_0-1, r,1} \subset \F \label{tsr2}\\
\text{ for all  } (i_1, \cdots, i_{s_0}) \text { such that }(X_1^{i_1} \cdots X_{s_0-1}^{i_{s_0-1}})
\in Mon_{ s_0-1}^{ r}.\nonumber
\end{eqnarray}
The  induction assumption implies that $P^r (X_0, \cdots, X_{s_0-1})$ can be defined algorithmically such that (\ref{tsr2}) holds.

Now  we will construct polynomials $P ^{r-1}$, $P ^{r-2}, \cdots ,P  ^ 0$ inductively from (\ref{tsr}), (\ref{tsr1}) and (\ref{tsr2}). 
Assume this has been done for all $P^r, \cdots, P^{r-k+1}, \, 1\le k\le r+1$.  Substituting $X_{s_0} = k$ into (\ref{tsr1}) and comparing this with (\ref{tsr}), we obtain the following defining equation for $P^{r-k}: \F ^{s_0-1} \to \F$
\begin{eqnarray}
f_{S_{s_0, r}} (i_1, \cdots, i_{s_0-1}, k ) = P ^r (i_1, \cdots, i_{s_0-1}) + k P ^{r -1} (i_1, \cdots, i_{s_0-1}) + \cdots \nonumber \\
 + k ! P ^{r-k} (i_1, \cdots, i_{s_0-1}) = b_{i_1, \cdots,  i_{s_0 -1}, k}.\label{combi4a}\\
 \LLR P ^{r-k} (i_1, \cdots, i_{s_0-1}) = \beta^{r-k}_{ i_1, \cdots, i_{s_0-1}}\in \F\label{combi4}\\
 \text{ for all} (i_1, \cdots, i_{s_0-1})  \text { such that }  (X_1^{i_1} \cdots X_{s_0-1}^{i_{s_0-1}})
\in Mon_{ s_0-1}^{ r} \text{ and for }\nonumber\\
 \beta^{r-k}_{ i_1, \cdots, i_{s_0-1}}: = {1\over k !} [ b_{i_1, \cdots,  i_{s_0 -1}, k}-(P ^r (i_1, \cdots, i_{s_0-1})+\label{combi4b}\nonumber\\
   + k P ^{r -1} (i_1, \cdots, i_{s_0-1}) + \cdots + k!P^{r-k +1}(i_1, \cdots, i_{s_0-1}) )].\nonumber
\end{eqnarray}
By the induction assumption $P^{r-k}$ can be algorithmically constructed using (\ref{combi4}). This completes the induction step. Hence Proposition  \ref{Lemma 4.4} is valid for all $s$.
\end{proof}

\begin{corollary}[cf. Example \ref{mom}] \label{colin}  Assume that    $\{ b_1, \cdots, b_m\}$  are linearly independent  vectors  in $\F^m$. Then  there exists  a polynomial map $f : \F \to \F ^m$ of degree $m$ whose image  contains  the points $b_0 = 0,  b_1, \cdots, b_m$. 
In  other words, $f$ is is $(m-1, 1)$-elusive.
\end{corollary}
 
Let us consider the interpolation problem for homogeneous polynomial mappings. Since  each homogeneous polynomial $ f\in Pol^r _{hom} (\F^{n+1}, \F^m)\\ \subset Pol^r (\F^{n+1}, \F^m)$ is  defined uniquely by the value of its  restriction to  the hyperplane $b^{n +1} = 1$ in $\F^{n+1}$, we get immediately from Proposition \ref{Lemma 4.4}

\begin{corollary}\label{inthom} 
1. Given a  tuple $S_{s,r,m}$ of  $\mathrm{\binom{s+r}{r}}$ points  $b_{i_1, \cdots , i_s}$ in $\F ^m$, where $i_j \in \N$ and $\sum_{j=1}^s i_j \le r$, there exists
an algorithmically constructed homogeneous polynomial  mapping  $ f_{S_{s,r,m}} : \F ^{s+1} \to \F ^m$ of degree $r$ 
such that 
\begin{equation}
f_{S_{s,r,m}} (i_1, \cdots, i_s, 1) = b_{i_1, \cdots , i_s} \label{interp1}
\end{equation}
for all $(i_1, \cdots, i_s)$ satisfying $i_j \in \N$ and $\sum_{j =1} ^s i_j \le r$.

2. Let  us abbreviate $\binom{s+r}{r}$ by $b(s,r)$.  Proposition \ref{Lemma 4.4} and the formulas  in its proof give a linear isomorphism
\begin{equation}
I^{b(s,r)}_{m }: \F ^{m b(s,r)} \to  Pol^r (\F ^s, \F ^m),\label{inter2}
\end{equation}
which associates any point $\overline{S_{s,r,m}} \in \F ^{m b(s,r)} $ with a polynomial mapping $f_{S_{s,r,m}}\in Pol^r (\F ^s, \F ^m)$.
\end{corollary}

 Proposition \ref{Lemma 4.4}  motivates the following
 
\begin{definition}\label{strel} A mapping $f\in Pol ^p (\F ^n, \F ^m)$ is called {\it strongly $(s,r)$-elusive}, if  the set  $\{ f (i_1, \cdots, i_n)|  i_j \in \N \text { and } \sum _{j =1} ^ n i_j \le p\} $ is  $(s,r)$-elusive.
\end{definition}

\begin{remark}\label{elstr} 1.  By Lemma \ref{ev}, a polynomial mapping $f \in Pol^p(\F ^n, \F ^m)$ is strongly $(s,r)$-elusive, if and only if
the point in $\F ^{mb(s,r)}$ associated with the  tuple $(f(i_1, \cdots, i_n)|  i_j \in \N \text { and } \sum _{j =1} ^ n i_j \le p) $
does not belong to the image of the evaluation mapping $Ev _{r,s,m}  ^{b (p,n)}$.

2. A strongly $(s,r)$-elusive polynomial mapping $f \in Pol ^p (\F ^n, \F ^m)$ is $(s,r,k)$-elusive for any $k \ge \binom{n+p}{p}$, and, hence, it is $(s,r)$-elusive.

3.  If $f\in Pol ^p (\F ^n, \F^m)$ is  (strongly) $(s,r)$-elusive, then it  is  (strongly) $(s', r)$-elusive  for any $s'  \le  s$.

\end{remark}

\section{Zariski closure of the image of a polynomial mapping, effective elimination theory and algebraic number field theory}

 Remark \ref{elstr} asserts that   
a verification of the strong $(s,r)$-elusiveness  of a polynomial mapping $f$ can be reduced
to the following problem. Given a  polynomial map $\tilde f : \F ^ n \to \F^m$ and  given a point  $b \in \F^m$, 
verify whether  $b$  belongs to the image $\tilde f (\F^m)$.  This problem is in fact a part of the elimination theory,  which we discuss in this section (Lemma \ref{ker}, Corollary \ref{cloim}). We   develop the method invented by Kumar-Lokam-Patankar-Sarma \cite{KLPS2010} that uses effective elimination theory   combined with  algebraic number field theory in order to get  concrete points  $b$ which do not belong to the Zariski closure of the image of a polynomial mapping $\tilde f$, if $\tilde f$ is defined over $\Q$, such that   the coordinates of $b$ are algebraic numbers (Proposition \ref{glok}). This result will be used in  the next section  to find a sufficient condition for a polynomial mapping $f$ to be  strongly $(s,r)$-elusive. As a consequence, we will   construct  in the next section concrete polynomial mappings and multivariate polynomials  whose  circuit  size  is large.  We note that the idea  to use  algebraic numbers to construct   polynomials which are hard  to compute  first appeared in the works  by Strassen-Schnorr and Heintz-Sieveking   (Remark \ref{polq}.1).

Given a polynomial mapping $f \in Pol( \F^{n} , \F^{n+k})$, where  $\F = \R$ or $\F = \C$ and $k \ge 1$, we are interested in the image $f (\F^n) \subset \F^ {n +k}$. 
There are also several available methods to detect whether a point $b$  belongs to   the Zariski closure $\overline{f (\F ^ n)}$ of $f(\F^n) \subset \F ^ {n+k}$, based on  algebraic description of   the ideal of the sub-variety $\overline{f(\F^n)}$.    
The polynomial mapping  $f = (f^1, \cdots, f^{n+k})$ induces  a ring homomorphism
$$ f ^ * : \F[Y_1, \cdots, Y_{n +k}] \to \F [X_1, \cdots, X_n],\quad Y_i \mapsto f^i(X_1, \cdots, X_n) .$$
Denote by $I(f (\F ^ n))$ the ideal  of $f (\F ^ n)$ (i.e. the ideal  of all polynomials on $\F^m$ which vanish on $f(\F^n)$).

\begin{lemma}(\cite[Proposition 15.30]{Eisenbud1994}, \cite[Lemma 1.8.16]{G-P2007}) \label{ker} Assume that $f$ is a polynomial mapping from $\F^n$ to $\F^{n+k}$. Then\\
1. $\ker f^* = I (f (\F ^ n)) = I(\overline{f (\F ^ n)})$.\\
2. Let $I$ be  the ideal in $\F[X_1, \cdots,X_n,Y_1, \cdots, Y_{n+k}]$ generated by $\{ Y_1 - f^1, \cdots, Y_{n+k} - f^{n+k}\}$. Then
$$\ker f ^ * = I  \cap \F [Y_1,\cdots, Y_{n+k}] .$$
\end{lemma}

\begin{remark}\label{cloim} Let $f: \F ^n \to \F^m$ and $g: \F^s \to \F^m$ be two polynomial mappings.  Clearly, $f(\F^n) \not \subset g(\F^s)$, if
$\overline{f(\F^n)} \not\subset \overline{g(\F^s)}$, equivalently by Lemma \ref{ker}, if  $\ker  f ^* \not \supset \ker g^*$.
\end{remark}

In general it is hard to find explicitly an element in $\ker f^*$. We know   only  algorithms  for   determining the generators of $\ker f ^ * = I  \cap \F [Y_1,\cdots, Y_{n+k}]$  based on Gr\"obner's basis 
or  on resultants for  determining  a  special  element  of $\ker f ^ *$  of  the corresponding system of polynomials, see e.g. \cite{G-P2007}. 
These algorithms  are time-consuming, and they do  not  give us  any partial knowledge  of the generators of $\ker f^*$ at the first glance.  In \cite{KLPS2010}  Kumar, Lokam, Patankar and Sarma  used a result in effective elimination theory to get partial knowledge of an element in $\ker f^*$ and   combining this knowledge with  algebraic number field theory they obtained concrete matrices with high rigidity.  
Our extension of their method  also uses  the  same result  in  effective elimination theory, namely the following

\begin{lemma} (\cite[p.6 Theorem 4]{BMMR2002}) \label{efeli} Let $I =\la f^1, \cdots, f^s\ra$   be an  ideal in the polynomial ring $\F [Y_1, \cdots, Y_m]$ over an infinite field $\F$.  Let  $d$ be the maximum total degree of the generators $f^i$.  Let $Z = \{ Y_{i_1}, \cdots, Y_{i_l}\}$ be a subset  of indeterminates $\{ Y_1, \cdots, Y_m\}$. If $ I \cap \F [Z] \not = 0$ then  there exists  a non-zero polynomial $g \in I \cap \F[Z]$ such that $g = \sum _{i =1} ^s  g^i f^i$ with $g^i \in \F [Y_1, \cdots, Y_m]$ and $\deg (g^if^i) \le (\mu +1)(m +2) ( d^\mu +1) ^{ \mu +2}$ for $i \in [1,s]$, where $\mu = \min \{s, m\}$.
\end{lemma} 

Set $D(m,r) =  (m+1)(m+2) ( r^m  +1) ^{m+2}$.

\begin{remark} \label{imp} Applying  Lemmata \ref{ker} and \ref{efeli} to the ideal  $I =\la  Y_1 - f^1, \cdots, Y_{n+k} - f^{n+k}\ra $,
 and to $Z = \{ Y_1,\cdots, Y_{n+k}\}$, observing that $ I \cap \F [Z] \not = 0$ if $k\ge 1$,  we  obtain  
 the existence of a  polynomial $ g \in \ker f ^* = I \cap \F [Z]$  whose degree  is less  than or equal  $D(n+k,\deg f)$. 
 Here $\deg f$ is the total degree  of the generators $f^i$. Thus, to prove that a point $b\in \F^{n+k}$ does not belong to the image  $f (\F^n)\subset \F^{n+k}$ it suffices to show that $g (b) \not = 0$  for any  $g \in Pol ^{D(n +k ,\deg f)}(\F ^n)$.  
\end{remark}

 To find  such a point  $b\in \F^{n+k}$  we use the algebraic number field theory, assuming $\F = \R$ or $\F = \C$, and that  $f$ is defined over $\Q$, i.e.,  all polynomials  $f^i$ in question are defined over $\Q$. 
\medskip
The following Proposition  is a  generalization of \cite[Theorem 8]{KLPS2010}.
 
 \begin{proposition}\label{glok}   Let  $s \le m-1$ and $ f : \F^s \to \F^ m$ be a polynomial mapping over $\Q$  of degree $r$. 
 
 1.  Assume that $p_1, \cdots, p_{s+1}$ are distinct  prime numbers  such that $p_i \ge D (m,r)+2$ for all $i$.   Set 
 $$b ^i : = e ^{ 2\pi \sqrt{-1}\over p_i}\text{ and } \tilde b ^i : =\sum_{j = 1} ^i a_j^ib ^j$$ 
where $ a_j^i \in \Q$ and $a^i _i \not = 0$.    Then
 $\tilde b = (\tilde b^1, \cdots, \tilde b ^{s+1}, a ^{s+2}, \cdots, a ^m)\in \C^m$  does not belong to the   image of $f$  for $\F = \C$ and  for any $(a^{s+2}, \cdots, a^m)\in \Q^{m-s}$.
 
 2. Assume that $p_1, \cdots, p_{s+1}$ are distinct  prime numbers  such that $p_i \ge 2D (m,r)+3$ for all $i$.  Set 
 $$b ^i : = e ^{ 2\pi \sqrt{-1}\over p_i}\text{ and } \tilde b ^i : =\sum_{j = 1} ^i a_j^i(b ^j  + \overline{b^j})$$ 
 where  $ a_j^i \in \Q$ and $a^i _i \not = 0$.  
Then
 $\tilde b = (\tilde b^1, \cdots, \tilde b ^{s+1}, a ^{s+2}, \cdots, a ^m)\in \R^m$  does not belong to the   image of $f$  for $\F = \R$ and  for any $(a^{s+2}, \cdots, a^m)\in \Q^{m-s}$.
\end{proposition}

\begin{proof}  Proposition \ref{glok} is a consequence of  Lemmas  \ref{ker}, \ref{efeli}, Remark \ref{imp} and Proposition \ref{modi} below.
\end{proof}

\begin{proposition}\label{modi} 
  Assume that $p_1, \cdots, p_m$ are distinct  prime numbers  such that $p_i \ge D +2$ for all $i$.    Set $b ^i : = e ^{ 2\pi \sqrt{-1}\over p_i}$ and $\tilde b ^i : =\sum_{j = 1} ^i a_j^ib ^j$, where $ a_j^i \in \Q$ and
  $a ^i _i \not = 0 $.  
  
  1. Then  for all $g \in Pol ^D (\Q^m)\subset Pol ^D (\C^m)$ we have
 $$g(\tilde b^1, \cdots, \tilde b ^{m})\not = 0.$$
 
 2. Then for all $g \in Pol ^{\lfloor {D +1\over 2} \rfloor} (\Q ^m)\subset Pol ^{\lfloor {D+1 \over 2} \rfloor} (\R ^m)$  we have
 $$ g ( Re (\tilde b^1) , \cdots, Re(\tilde b ^{m}))\not = 0,$$
 where $\lfloor {D+1\over 2}\rfloor$  denotes the integral part of ${D+1 \over 2}$, and $Re (a)$ denotes the real part of $a\in \C$.
\end{proposition}

\begin{proof} 1.  Let us  prove  Proposition \ref{modi}.1 by induction on $m$.   For $m = 1$ this is trivial, since  $[\Q(\tilde b^1) : \Q] = p_1-1 \ge D+1$. 

Now suppose that
the statement is true when the number of variables of  a polynomial $g$ is strictly less than $m$. Assume  that
the statement is not true for $m$, i.e. there exists $g\in Pol ^D (\Q ^m)\subset Pol^D (\C^m)$ such that
\begin{equation}
g(\tilde b^1, \cdots, \tilde b ^m ) = 0.\label{oppos}
\end{equation}

Let us write
$$g (Y_1, \cdots, Y_m) = \sum_{i = 0} ^d g_i(Y_1, \cdots, Y_{m-1}) Y^{d-i}_m,$$
where $g_i \in \Q[Y_1, \cdots, Y_{m-1}]$, since $g$ is defined over $\Q$.
If $g$ does not depend on $Y_m$, or equivalently $g_i = 0$ for $i \in[0, d-1]$, the induction assumption implies that
the induction statement  is also valid for $m$, since  $g =g_d \in Pol ^D (\Q ^{m-1})\subset Pol ^D (\C ^{m-1})$ satisfies 
$$g(\tilde b^1, \cdots, \tilde b ^m ) \not= 0.$$

Thus,   we can assume that  $Y_m$ enters in $g$. Hence
$$g (\tilde b ^1, \cdots , \tilde b ^{m-1}) (x) \not = 0 \in \Q(\tilde b^1, \cdots, \tilde b ^{m-1})[x].$$
Clearly, (\ref{oppos}) implies that  $\tilde b ^m$ is a root of  a non-zero polynomial in  one variable of degree $D$ over
the extension $\Q(\tilde b^1, \cdots, \tilde b ^{m-1})$.
Thus
\begin{equation}
[\Q(\tilde b ^1, \cdots, \tilde b ^m): \Q(\tilde b ^1, \cdots, \tilde b ^{ m -1})] \le D. \label{oppo2a}
\end{equation}
Since $\tilde b ^i  =\sum_{j = 1} ^i a_j^ib ^j$, where $ a_j^i \in \Q$ and $a ^i _i \not = 0$, we have
 $$\Q(\tilde b ^1, \cdots, \tilde b ^k)  = \Q (b^1, \cdots, b^k)\text { for all } k \le m.$$
Thus (\ref{oppo2a}) implies that
\begin{equation}
[\Q( b ^1, \cdots,  b ^m): \Q( b ^1, \cdots,  b^{ m -1})] \le D. \label{oppo2}
\end{equation}  
Since $\Q(b ^m)$ is  a Galois  extension of $\Q$, applying \cite[Theorem 1.12 p. 266]{Lang2002}  we obtain   
\begin{equation}
[\Q( b ^1, \cdots,  b ^m): \Q( b ^1, \cdots,  b^{ m -1})] = [\Q(b ^m):\Q] = p_m -1 \ge D +1. \nonumber
\end{equation} 
Thus, (\ref{oppo2}) does not hold. The contradiction  implies that   Proposition \ref{modi}.1 is also valid for $m$. This completes the proof of Proposition \ref{modi}.1.

2. Now let us prove Proposition \ref{modi}.2.   Repeating the argument in the proof of Proposition \ref{modi}.1 we  derive Proposition \ref{modi}.2 from the following 

\begin{lemma}\label{galoisr} For $ 1\le i \le m$, $\Q(Re (\tilde b^i))$ is a Galois extension of $\Q$  and
$[\Q(Re (\tilde b ^i)): \Q ] \ge \lfloor {D+1 \over 2}\rfloor$.
\end{lemma}

\begin{proof} Since $\Q (Re (\tilde b ^i))$  is a subfield of  the Galois extension $\Q (b^i)$, whose Galois group is cyclic, $\Q(Re (\tilde b ^i))$ is also a Galois extension. 
Note that the Galois group $G_{\Q (Re (\tilde b ^i))}$ of $\Q (Re (\tilde b ^i))$  is  $\Z_{p_i -1} / \Z_2$.  Hence
$$[Q(Re (\tilde b ^i)): \Q ] \ge  \# (G_{\Q (Re (\tilde b ^i))})   = {p_i -1\over 2}\ge  \lfloor {D + 1\over 2}\rfloor.$$
This proves  Lemma \ref{galoisr}.
\end{proof}
This completes the proof of Proposition \ref{modi}.
\end{proof}

\begin{remark}\label{polq} 1. One of the main  ideas   of the Kumar-Lokam-Patankar-Sarma method,  adapted and developed  to our case,  is to relate the separable degree  of the  field extension  $\Q(\alpha_1, \cdots, \alpha_n)$, where $\alpha_i $ are algebraic numbers, with  the complexity  of polynomials and polynomial mappings whose  monomial coefficients are $\alpha_i$. This  idea  has been  invented   before by  Strassen-Schnorr  and  Heintz-Sieveking. We refer the reader to \cite[Chapter 9]{BCS1997} for  exposition  of their methods. Their  technique  is used to construct   polynomials   $P_n$ in one variable  of degree $n$ of  multiplicity  complexity   with lower bound of type $n^ a$,  $a<1$,  where  the coefficients of $P_n$ are algebraic  numbers. Our technique is used, in particular, to construct   ($poly(n)$-definable) multivariate polynomial mappings and polynomials of constant  degree, whose   (constant-depth) circuit size  is high (Propositions \ref{c45}, \ref{super}, Corollaries \ref{baust}, \ref{cor:circs}).

2. Let $f : \F ^n  \to \F^m$ be a mapping. The question whether $f$ is a polynomial mapping defined over $\Q$   depends on the choice of  a basis $(V_1, \cdots, V_n)$ of $\F^n$ as well as on the choice
of  a basis $(W_1, \cdots, W_m)$ of $\F^m$.  Assume that $f: \F ^n \to \F^m$ is a polynomial mapping defined over $\Q$ with respect to a basis  $(V_1, \cdots, V_n)$ of $\F^n$ and a basis $(W_1, \cdots, W_m)$ of $\F^m$.
Then $f$ is also a polynomial mapping defined over $\Q$  with respect a basis $(V_1', \cdots, V_n')$ of $\F^n$ and a basis $(W_1', \cdots, W_m')$ of $\F^m$, if $V_i' = \sum_jA_{ij} V_j$, \, $W _i'= \sum B_{i'l} W_l$ and
$A_{ij}, B_{i'l}$ are rational numbers.  In other words, the basis $(V_i')$ (resp. $(W_j')$)  is obtained from the basis  $(V_i)$ (resp. $(W_j)$) by a linear transformation over $\Q$.

3. The  set of all transformations $(a^ i_j)\in Mat _n (\Q)$ with $a^i_j = 0$ if $j > i$ and $ a^i _i \not = 0$, which  enter in Proposition \ref{glok},  forms the   solvable group $B_n (\Q)$. 
\end{remark}

\section{Examples and applications}

In this section, using the methods developed in the previous sections,  we construct concrete examples of $(s,r)$-elusive functions (Proposition \ref{pola}, \ref{super}). 
As a result,   we construct a sequence  of 
$poly(n)$-definable polynomial mappings $P_n : \F^{2n} \to \F^n$ of  constant degree $5r+1$ whose depth-$r$ circuit size  is greater  than $n^2/(50 r^2)$, and  consequently,  a sequence of multivariate $poly(n)$-definable polynomials of constant degree $5r+2$ whose  depth-$\lfloor r/3 \rfloor$ circuit  size is  greater than $n^2/250 r^2$ (Proposition \ref{c45}, Corollary \ref{baust}). We compare this result  with   similar  results (Remark \ref{r2}).  We  also  construct  a sequence  of  elusive polynomial mappings, whose  monomial  coefficients  are algebraic numbers, which  give   polynomials  with algebraic number coefficients such that   their circuit  size  is very large (Corollary \ref{cor:circs}).

To apply the effective  elimination theory to  elusive functions, we need to estimate   the degree  of the evaluation mapping.

 \begin{lemma}\label{doveq}  The evaluation map $Ev^k_{r,s,m}$, defined in (\ref{evkrsm}), is of total degree
$r +1$,  it is also defined over $\Q$. 
\end{lemma}

\begin{proof}
Let us compute the degree of the evaluation map $Ev ^k_{r,s,m}$.  Let $\{ V_j, \, 1\le j \le s\}$ be a basis of $\F ^{s}$. 
Let $\{( X_1 ^{i_1}\cdots X_s ^{i_s})|\, \sum_{j=1}^s i_j \le r\}$ be the basis consisting of  monomials in $Pol ^r (\F ^s)$.  Let $f= (f^1, \cdots , f^m) \in Pol ^r ( \F^s, \F ^m)$  where
$$f ^l : =  \sum_{0\le i_1+ \cdots + i_s\le r} a_{i_1 \cdots i_s, l}(X_1 ^{i_1}\cdots X_s ^{i_s}). $$ 
Let   $b=  (b_1 , \cdots, b_k) \in  (\F ^s) ^k$ where
$$ b_i = \sum_j b_i ^j V_j\in \F ^{s}.$$
Then
\begin{equation}
Ev^k_{r,s,m} (f, b) = (f (\sum_{j=1} ^sb_1^j V_j), \cdots , f (\sum _{j=1}^s b_k^j V_j))\in (\F ^m ) ^k. \label{p141}\\
\end{equation}
Clearly $Ev^k_{r,s,m}$ is a polynomial mapping, whose degree does not depend on $k$ or on $m$.  Note that for $k = 1$ and $m =1$ we have
\begin{equation}
Ev^1_{r,s,1} (f, b) =\sum_{0\le i_1+ \cdots +i_s\le r}  a_{i_1 \cdots i_s} (b_1 ^{1})^{i_1}\cdots (b_1^{s}) ^{i_s}\in \F.\label{evq}
\end{equation}
(\ref{evq}) implies  that $Ev ^1_{r,s,1}$ is of degree 1 on $f$  and of maximal degree $r$  on $b$. This proves the second assertion of Lemma \ref{doveq}.
\end{proof}

Next,  we need    a choice  of a  basis of the space
$Pol^r(\F^n)$ which is not monomial.

\begin{definition} A polynomial $(X- i)(X-i +1) \cdots X\in \F [X]$ is called {\it a pseudo-monomial}, if  $i \in \N$. A constant is also called  a pseudo-monomial. A  polynomial
$f \in \F [ X_1, \cdots, X_n]$ is called {\it a pseudo-monomial}, if  $f = f ^1\cdots f^n$, where,  for $1 \le i \le n$,  $f^i\in \F [X_i]$ and $f ^i$ is a pseudo-monomial.  
\end{definition}

\begin{remark}\label{pseu1} 1.  According to the  lexicographical ordering in $Pol ^p (\F ^n)$  the  linear transformation 
$Pol ^p (\F ^ n) \to Pol  ^ p (\F ^n)$ sending the basis consisting of pseudo-monomials to the standard basis of monomials is  an element of  the solvable group $B_{\binom{n+p}{p}} (\Q)$. In particular,  any polynomial $f \in Pol ^ p (\F ^ n)$  can  be written  in a unique way as a linear combination  of pseudo-monomials.

2. The notion of pseudo-monomials is motivated by  the interpolation formulas (\ref{4.1.1}), (\ref{4.1.1.a}), (\ref{tsr}), (\ref{combi4}), (\ref{combi4b}) for  polynomial mappings.
Using these formulas we have defined the coefficients  $\lambda ^i _{i_1, \cdots i_n}$ of the pseudo-monomials $(X_1- i_1)(X_1-i_1 +1) \cdots X_1 (X_2-i_2)\cdots X_2\cdots(X_n-i_n) \cdots X_n$ in the component
$f ^i$ of a polynomial mapping $f : \F ^n \to \F ^m$ as a rational linear combination  of the coordinates of the given points $b_{i_1\cdots i_m} \in \F ^m$.
\end{remark}

Next, we need the following

\begin{lemma}\label{kcondit} Assume that  $1\le s < m$. Then there exists a $(s,r)$-elusive  $K$-tuple in $\F^m$, if 
\begin{equation}
 K\ge \frac{m \binom{s+r}{s}+1}{m-s}.\label{ksrm}
 \end{equation}
\end{lemma}

\begin{proof} Note that
$$ \dim  (Pol^r(\F ^s, \F ^m) \times (\F ^s)^K)  = m\binom{s+r}{s} + sK. $$
It follows that  the  image of the  evaluation map  $Ev^K_{s,r, m}$ is a  proper  subset of co-dimension at least 1 in  $\F ^{m K}$ if $K$ satisfies (\ref{ksrm}).
Taking into account Lemma \ref{ev},  we obtain immediately Lemma \ref{kcondit}.
\end{proof} 

Using the interpolation formula in Proposition \ref{Lemma 4.4} we shall construct from 
$(s,r)$-elusive $K$-tuples  in $\F^m$  $(s,r)$-elusive  polynomial mappings $f : \F ^n \to \F^m$.
Given $K$ satisfying (\ref{ksrm}), let us assume that   two  positive integers $n, p$ satisfy the following conditions
\begin{equation}
\binom{n+p}{n} \ge  K \ge \frac{m \binom{s+r}{s} + 1 }{m -s}.
\label{elu1}
\end{equation}
By Proposition \ref{Lemma 4.4}, the first inequality in (\ref{elu1}) is a sufficient condition for the existence  of a polynomial mapping $f \in Pol ^p (\F ^n, \F ^m)$ such that the image  $f (\F^n)$ contains a given $K$-tuple in $\F^m$.

\begin{proposition}\label{pola}  Assume that $n, p$ satisfy (\ref{elu1}) with $K = \binom{n+p}{n}$.  Let $\Bb$  be   either  the monomial basis or the pseudo-monomial basis  of  the space $Pol^p(\R^n) \subset Pol^p(\C^n)$.  

1. Assume that $f^1, \cdots , f^m$ are polynomials in $Pol ^p (\C ^n)$ such that the   coefficients of each $f^j$  w.r.t. the basis $\Bb$, according to the lexicographical ordering, and beginning with the smallest term,  are 
$$ e ^{ {2\pi\sqrt{-1}\over p_1 ^j}}, \cdots , e ^{ {2\pi\sqrt{-1}\over p_K^j }} $$
 where $\{ p_i  ^j,\, 1\le i\le K,  \, 1\le j \le m\}$ are distinct 
prime numbers such that $p_i^j \ge D(m,r) +2$. Then the polynomial mapping $f = (f ^1 , \cdots, f^m): \C ^ n \to \C^m$ is  $(s,r)$-elusive.

2. Assume that $f^1, \cdots , f^m$ are polynomials in $Pol ^p (\R ^n)$ such that the  coefficients of each $f^j$ w.r.t. the basis $\Bb$, according to the lexicographical ordering, and beginning with the smallest term,  are 
$$Re( e ^{ {2\pi\sqrt{-1}\over p_1 ^j}}), \cdots , Re ( e ^{ {2\pi\sqrt{-1}\over p_K^j }}) ,$$
 where $\{ p_i  ^j,\, 1\le i\le K,  \, 1\le j \le m\}$ are distinct 
prime numbers such that $p_i^j \ge2 D(m,r) +3$. Then the polynomial mapping $f = (f ^1 , \cdots, f^m): \R ^ n \to \R^m$ is  $(s,r)$-elusive.  
\end{proposition}

\begin{proof} It suffices to show that  the  polynomial mappings $f$  defined in Proposition \ref{pola} are strongly $(s, r)$-elusive.   Equivalently, we need to show that  the set
$$S_K : = \{ f(i_1, \cdots, i_n)| \,  i_j\in \N \text { and  } \sum_{j =1 } ^n i_j \le p \}\subset \F ^m, $$
$\F = \C$ or $\F = \R$, is a $(s, r)$-elusive $K$-tuple. 
We will show that the associated point $\overline{S_K}\in (\F ^m) ^K$  does not belong to the image   of the evaluation map 
$Ev ^K_{r,s,m}$. 
By  Lemma \ref{doveq}  the evaluation map $Ev ^K_{r,s,m}$ is   a polynomial  mapping of  degree $(r+1)$, moreover it is defined over $\Q$. Remarks \ref{polq} and \ref{pseu1}.2  imply that
 Lemma \ref{doveq} also holds  with respect to the basis  of $(\F ^{m})^K = (\F ^K) ^m = Pol ^p (\F ^n, \F ^m)$ that is induced  from the  basis of pseudo-monomials in $Pol ^p (\F ^n)$.
 Now we will apply Proposition \ref{glok}  to show  that $\overline{S_K}$ does not belong  to the image of $Ev ^k_{r,s,m}$; more precisely, we will verify  that the coordinates  of $\overline{S_K}$ with respect to the  pseudo-monomial basis
 in  $(\F ^K) ^m = Pol ^p (\F ^n, \F ^m)$ satisfy the conditions of Proposition \ref{glok}.  Using Remarks \ref{polq}.2 and \ref{pseu1}.1 it suffices to
 consider the case of $f\in Pol ^p (\F ^n, \F ^m)$ whose  pseudo-monomial coefficients are given   by the recipe in Proposition \ref{pola}.

 By the assumption of Proposition \ref{pola} the first $m$ coordinates of  $\overline{S_K} \in (\F ^K) ^m = Pol ^p
  (\F ^n, \F ^m)$ are the smallest  pseudo-monomials according to the  lexico-graphical ordering, i.e., they are field elements. These field elements  are    numbers 
 $$e ^{ {2\pi\sqrt{-1}\over p_1 ^1}} \cdots,  e ^{{2\pi\sqrt{-1}\over p_1 ^m}},$$
 if $\F = \C$. (The case $\F = \R  $ is  similar).
 Now assume  that  the conditions of Proposition \ref{glok} hold for the  first $lm$-coordinates of  $\overline{S_K} \in (\F ^K) ^m = Pol ^r (\F ^n, \F ^m)$, for $l \ge 1$.
 The interpolation formula 
 (\ref{combi4}) for the $(l+1)j$ coordinate $(S_K)_{l+1} ^j$ of $\overline {S_K}$, $ 1\le j \le m$, if $\F = \C$,  has the following form
$$(\overline{S_K})^{l+1}_j = a ^{l+1}_j   e ^{ {2\pi\sqrt{-1}\over p_{l+1} ^j}} +\sum_{1 \le k \le l}  a ^{ l +1,k}_j e ^{ {2\pi\sqrt{-1}\over p_{k} ^j}}, $$
where $a ^{l+1, k}_j \in \Q$  and $a ^{l+1}_j\not = 0$. (The case $\F = \R  $ is  similar).
Thus the conditions in Proposition \ref{glok} also hold for first $(l+1)m$-coordinates of  $\overline{S_K} \in (\F ^K) ^m = Pol ^r (\F ^n, \F ^m)$. 
 This completes the proof of Proposition \ref{pola}.
\end{proof}

\medskip

In \cite[\S 3.4]{Raz2009}  Raz proposed  a method for  constructing  polynomials
of large complexity  using  $(s,r)$-elusive functions.   Propositions \ref{r311}, \ref{c45} below are   sample applications of Raz's method. 

 Given  a  tuple of $n ^2$ function $ f_{ij}\in \F[X_1, \cdots, X_n]$,  $1\le i, j \le n$,
 we define  an $n$-tuple  of polynomials $\tilde f_i \in \F[X_1\cdots , X_n, Z_1, \dots, Z_n]$,
$i \in [1, n]$, as follows (cf. \cite[\S 3.3]{Raz2009})
\begin{equation}
\tilde f_i (X_1, \cdots, X_n, Z_1, \cdots, Z_n): =\sum_{j =1}^nf_{ji}(X_1, \cdots, X_n)Z_j\label{raze1}
\end{equation}

\begin{proposition} \label{r311}\cite[Proposition 3.11]{Raz2009} Let $n, r\le s$ be integers.
Let $f : \F ^ n \to \F ^{n^2}$ be a polynomial mapping. If $f$ is  $(s,r)$-elusive, then any depth-r arithmetic circuit
over $\F$ for the $n$-tuple  $\{ \tilde f_i: \F ^ {2n} \to \F, \, i \in [1, n]\}$ of  polynomials  defined by (\ref{raze1}) is of size
greater than $s$.
\end{proposition}

Using Proposition \ref{r311} and  our construction of $(s,r)$-elusive functions in Proposition \ref{pola},
we  shall construct  sequences of polynomials with large constant-depth circuit size.
 
\begin{proposition}\label{c45} Let $\F = \C$ or $\F = \R$, and $ 1\le r\in \N$ a constant.  There are infinitely many  sequences of $poly(n)$-definable  polynomial
mappings $\tilde f_{n,r}\in Pol ^{5r+1}(\F ^ {2n},\F ^n)$, which satisfy  the following properties. All the coefficients of $\tilde f_{n,r}$ are algebraic numbers, and any  (unbounded fanin) depth-$r$ arithmetic circuit over $\F$ for $\tilde f_{r,n}$ is of size greater than $ {n ^2\over 50 r ^ 2}$. 
\end{proposition}

\begin{proof} Let $n ' \ge r^2$ be an integer, and set
$$  n: = 5n'r,\,  p: = 5r,\,  m: = (n') ^ 2 ,\,  s := \lfloor (n') ^2/2\rfloor.$$
First we  will show  that  the chosen  values  $(n,p, m, s)$   satisfy   Condition  (\ref{elu1}).  Since $(m-s)\ge m/2$  it suffices to  show
\begin{equation}
\binom{5n'r + 5r}{5r} \ge 2\frac  {(n') ^ 2\binom{\lfloor{(n') ^2\over 2}\rfloor + r}{r} + 1}{(n') ^2}.
\label{eq1}
\end{equation}
Clearly (\ref{eq1}) is a consequence  of Lemma \ref{lem:inn1}, which  we  now prove.
\begin{lemma}\label{lem:inn1} We have
\begin{equation}
\binom{5n'r + 5r}{5r} \ge 2 [\binom{\lfloor{(n') ^2\over 2}\rfloor + r}{r} + 1].\label{eq2}
\end{equation}
\end{lemma}
\begin{proof} 
We rewrite the  LHS of (\ref{eq2})  as
\begin{equation}
\Pi _{k = 0} ^{r-1} \frac{(5n'r + 5k +1)(5n'r + 5k +2)\cdots (5n'r +5k+5)}{(5k +1)(5k +2) \cdots (5k +5)}, \label{eq2a}
\end{equation}
and RHS of (\ref{eq2}) as
\begin{equation}
 2(\Pi_{k =1} ^r  \frac{\lfloor{(n') ^2\over 2}\rfloor +k}{k} +1).\label{eq2b}
 \end{equation}
Using (\ref{eq2a}) and  (\ref{eq2b}),     taking into account the   following inequalities
$$ 2(\Pi_{k =1} ^r  \frac{\lfloor{(n') ^2\over 2}\rfloor +k}{k} +1)\le \Pi_{k =1} ^r  (\frac{(n') ^2 +2k}{k} +2)\le ((n') ^2 +4)^r, $$ 
$$ \frac{(5n'r + 5k +1)(5n'r + 5k +2)\cdots (5n'r +5k+5)}{(5k +1)(5k +2) \cdots (5k +5)}\ge (\frac{5n'r + 5k +5}{5k+5})^5$$
$$\ge (\frac{(n'+1) r}{r})^5 \text{ (since $k +1 \le  r$)},$$
 to prove  Lemma \ref{lem:inn1} it suffices  to establish the following inequality
\begin{equation}
(n'+1)^5 \ge  (n')^2 + 4.\label{eq:inn2}
\end{equation}
Clearly (\ref{eq:inn2}) holds, since  $n'  \ge 1$.
This completes the proof of Lemma \ref{lem:inn1}.
\end{proof}
 Since (\ref{eq1}) is  fulfilled,  Proposition \ref{pola} implies that   there exists a $(s,r)$-elusive function $f_{n,r}'\in Pol ^{p}(\F ^n , \F ^m)$.
 
   We extend $f_{n,r}' $ to a polynomial mapping,  denoted by $ f_{n,r}$, from $\F^{ n }$ to $\F ^{n^2} $ by composing $f_{n,r}$ with the canonical 
   embedding $\F^{(n') ^ 2} \to \F ^ { n ^ 2}$. Clearly $f_{n,r}$ is also  $(s,r)$-elusive.
Since $r$ is fixed and all the coefficients of $f_{n,r}$  are given,  $f_{n,r}$ is $poly(n$)-definable.
Let $\tilde f_{n, r} : \F ^{2n} \to \F^n$  be the polynomial mapping   obtained from $f_{n,r}: \F^n \to \F^{n^2}$  by recipe (\ref{raze1}). 
Set 
$$\tilde f_{n, r}: = ((\tilde f_{n,r})^1, \cdots, (\tilde f_{n,r})^n),$$
where $(\tilde f_{n,r})^i$, $i \in [1, n]$, is the $i$-th coordinate of the   polynomial mapping $\tilde f_{n,r}$.
Since  $f_{n,r}\in Pol^{5r}(\F^n, \F^{n^2})$, we  have $\tilde f_{n,r} \in Pol ^{5r+1}(\F ^{2n}, \F ^ n)$.  Furthermore, $(\tilde f_{n,r})^i, 1\le i \le n,$ is $poly(n)$-definable, since $r$  is fixed.
Taking into account Proposition \ref{r311} this completes the proof of Proposition
\ref{c45}.
\end{proof}

\begin{corollary}\label{baust} Let $\tilde  f_{n,r} : = ((\tilde f_{n,r})^1, \cdots, (\tilde f_{n,r})^n) \in Pol ^{5r+1}(\F ^{2n}, \F^n)$ be the polynomial mappings  defined in Proposition  \ref{c45}.
Let $\hat f_{n, r} : \F ^{2n}\times  \F ^n  \to \F$ be defined by 
$$ \hat f_{n, r} (X_1,\cdots, X_{n},Z_1, \cdots, Z_n, Y_1, \cdots, Y_n) : = \sum _{ i =1} ^n (\tilde f _{n,r})_i (X_1, \cdots, Z_n) Y_i. $$
Then any depth-$\lfloor r/3 \rfloor$ arithmetic circuit  for $\hat f _{n,r}$ is of size greater than ${n ^2 \over 250 r^2}$.
\end{corollary}

\begin{proof} We use  Raz'  argument in \cite [Corollary 4.6]{Raz2009}.  Baur and Strassen proved that if a polynomial $\hat f$ can be computed by an arithmetic
circuit of size $s$ and depth $d$, then all partial derivatives of that polynomial can be computed
by one arithmetic circuit of size $5s$ and depth $3d$.
\end{proof}

\begin{remark}\label{r2}  In \cite[Lemma 4.1]{Raz2009} Raz proposed a combinatoric method  to construct a $([n ^ {1 + 1/(2r)}], r)$-elusive function of degree
$ 5 r$ from $\F ^ {5nr}$ to $\F ^{n ^ 2}$, if  $n$ is  prime  and $1 \le r \le (\log _2 n)/100$. As a result, Raz obtained a lower bound $n ^{ 1 + 1 / (2r)}$ for  the size of any depth-$r$ arithmetic  circuit computing  $\tilde f_n\in Pol ^{5r+1}(\F^{n(5r+1)},\F ^n)$ \cite[Corollary 4.5]{Raz2009}  and  a lower bound  $n ^{ 1 + 1 / (2r)}/5$  for  any depth-$\lfloor r/3 \rfloor $ arithmetic circuit computing  $ \hat f_n\in Pol ^{5r+1}(\F^{n(5r+2)})$ \cite[Corollary 4.6]{Raz2009}.  Note that his polynomials $\tilde f_i$ have coefficients taking  values in $\{0, 1\}$.  Raz's  results is  an improvement of Shoup's and Smolensky's result \cite{S-S1991}, which 
gives a lower bound of $\Om(dn^{1+1/d})$ for depth $d$ arithmetic circuits, for  explicit polynomials
of degree $O(n)$ over $\C$.   Shoup and Smolensky  used  algebraic independent numbers  and a sequence  of  rapidly growing integers of the form $2, 2^n, \cdots, 2 ^{n ^ {n-1}}$ to construct such polynomials. We also like to mention  better lower bounds for depth four homogeneous circuits, see  e.g. \cite{FLMS2013},  but  these constant  deep  circuits  have    lower  bound on the  fanin  at the bottom layer  of   product gates (and ours   do  not have  such a bound).
\end{remark}

Raz     also generalized   his  construction of polynomials  of large  circuit size in Proposition \ref{r311}  as follows \cite[\S 3.1, 3.3]{Raz2009}.  
We fix $m'$ to be the number of monomials of total degree exactly $r$ in $n$ variables, that is, $m' =\binom{n+r-1}{r}$
and we fix $m = m' \cdot n$. Let $M$ be the set of all monomials of total degree exactly $r$ in the variables $\{z_1, \cdots,  z_n\}$.  Let $h : M \to [1, m']$
be the lexicographic order of monomials. Let us denote by $Pol^p_{hom} (\F^n, \F^m)$ the space of homogeneous polynomial mappings  of degree $p$  from $\F^n$ to $\F^m$.
Given  a  homogeneous polynomial  mapping $f = (f_{1,1}, \cdots, f_{m',n}) \in Pol^p_{hom}(\F^n, \F^m) = (Pol ^p_{hom} [x_1, \cdots, x_n])^ m$  we define an $n$-tuple of  polynomials $\tilde f_1, \cdots, \tilde f_n \in \F[x_1, \cdots, x_n, z_1, \cdots, z_n]$ as follows (cf. (\ref{raze1}))
\begin{eqnarray}
\tilde f_i (x_1, \cdots, x_n, z_1, \cdots z_n): = \sum_{g\in M}f_{h(g), i}(x_1, \cdots, x_n) \cdot g = \nonumber \\
=\sum_{j =1} ^ {m'} f_{j, i}(x_1, \cdots, x_n)h^{-1}(j).\label{raze2}
\end{eqnarray}
Now we define a polynomial $\tilde  f \in \F[x_1, \cdots, x_n, z_1, \cdots, z_n, w_1, \cdots, w_n]$  using (\ref{raze2}) and the following   formula (cf. the formula in Corollary \ref{baust})
\begin{equation}
\tilde f= \sum_{i=1}^n  w_i \cdot \tilde f_i. \label{raz3} 
\end{equation}

\begin{lemma}\label{lem:raz38} (\cite[Corollary 3.8]{Raz2009}) Let $1\le r \le n \le s$, and $ m = n \cdot \binom{n+r-1}{r}$ be  integers.
Let $f\in Pol^p( \F^n , \F^m)$ be a polynomial mapping.  If $f$ is $(s, 2r-1)$-elusive,  then any arithmetic circuit for the polynomial $\tilde f: \F^{3n} \to \F$  
constructed by recipe (\ref{raz3}) is of size $\ge \Om(\sqrt{s}/r^4)$.
\end{lemma}

In  \cite{Raz2009} Raz did not specify  the  value $\Om(\sqrt{s}/r^4)$  but it  is not   hard to   find  that value  using  Raz's results  in \cite{Raz2009}.  In \cite{Le2013}  we developed Raz's method, in particular  we  specified  the  lower bound  for  the circuit  size   of $\tilde f$, see e.g. \cite[Proposition 4.3]{Le2013}  for a slightly generalized  assertion.

Now  we   shall apply  Lemma \ref{lem:raz38}   and  our  methods    to  construct  polynomials   with  algebraic  number  coefficients  with  large circuit  size.  First  we need the following

\begin{proposition}\label{super}  Given $4\le r'\in \N$,   for $  n \in \N$ set 
$$ s(n): = (\lfloor{n\over (r'-1) r'}\rfloor) ^{r'-3}, \, m(n): = n \cdot \binom{n-1 +r'}{r'}, \, p = (r'-1) (2r'-1).$$
Then, for $\F = \C$ or $\F = \R$, if  $n \ge 2 (r'-1) r'$ and $ (n + r -4) ^4 \ge r ! $ there exists a   polynomial mapping $f \in Pol ^p( \F ^n , \F ^{m(n)})$ such that $f$ is $(s(n), 2r'-1)$-elusive, moreover  the monomial coefficients of
$f$ are algebraic numbers.
\end{proposition}

\begin{proof} Set $r : = 2r'-1$. 
  We will show that  $(n, p, m= m(n), s, r)$ defined in Proposition \ref{super} satisfy (\ref{elu1}) for $K : = \binom{n+p}{p}$, i.e., we need to verify that
\begin{equation}
\binom{n+p}{n} \ge \frac{m\binom{s+r}{s} +1}{m -s}.\label{super1}
\end{equation}
Since  $ (n + r -4) ^4 \ge r ! $ we get
\begin{equation}
(r')!\cdot  n  ^{r'-4}\le (n+r -4) ^ 4 \cdot n ^{r'-4} < n \cdot (n +1) \cdots   \cdots (n + r' -1).\label{eq:super1a}
\end{equation}
Since  $ 4\le  r'$ and $ n \ge 2$, taking into account (\ref{eq:super1a}),   we obtain
\begin{equation}
s(n) < n ^{r'-3} \le {n\over 2} n ^{r'-4} \le {n\over 2}\cdot \binom{n-1 + r '}{r'} \le { m +1\over 2} .\label{super2}
\end{equation}
Abbreviating $s(n)$ as $s$,   we deduce from (\ref{super2})
\begin{equation}
\frac{m\binom{s+r}{s} +1}{m -s}\le \frac{(m +1) \binom{s+r}{s}}{m-s} < \frac{(m+1)\binom{s+r}{s}}{m-{m+1 \over 2}} \le  2 (1 + {2\over m-1}) \binom{s+r}{s}.\label{super3}
\end{equation}
Clearly (\ref{super1})  follows from (\ref{super3}) and the following inequality
\begin{equation}
\binom{n+p}{n} \ge  2 (1 + {2\over m-1}) \binom{s+r}{s},\label{super31}
\end{equation}
which we now prove.  We rewrite  the LHS of (\ref{super31}) as
\begin{equation}
\Pi_{k =0} ^{r-1}\frac{(n + (r'-1)k +1)(n + (r'-1)k +2) \cdots (n + (r'-1) (k +1))}{((r'-1)k +1)((r'-1)k +2)\cdots (r'-1) (k +1)}\label{super31a}. 
\end{equation}
Since $p = (r'-1)(2r'-1) = (r' -1)  r$, we rewrite  the RHS of (\ref{super31}) as
\begin{equation}
 2 (1 + {2\over m-1}) \Pi_{k =1}^r \frac{s + k}{k}.\label{super31b}
 \end{equation}

\begin{lemma}\label{lem:ad}  For all $0\le k \le r-1$ we have
\begin{equation}
\frac{(n+1) ^{r'-1}}{(k +1) ^{r'-1}}   \ge 2\frac{ s +k +1} {k +1}.\label{eq:no}
\end{equation}
\end{lemma}

\begin{proof}  To prove  Lemma \ref{lem:ad} it suffices to     establish the following  inequality
\begin{equation}
(n+1) ^{r'-1} \ge  2\cdot r^{r' -2} \cdot (s + 2r' -1).\label{eq:ad1}
\end{equation}
Since $r \ge 7$ we have
\begin{eqnarray}
\frac{s + 2r' -1}{r} < {s\over 2}={1\over 2} (\lfloor{n\over (r'-1) r'}\rfloor) ^{r'-3} <\nonumber\\
{1\over 2} ({n\over (r'-1) r'})^{r'-1}< {1\over 2}({ n +1 \over r}) ^{r'-1}.\label{eq:ad2}
\end{eqnarray}

Clearly (\ref{eq:ad2}) implies   (\ref{eq:ad1}). This  completes  the proof  of Lemma \ref{lem:ad}.
\end{proof}

Using (\ref{super31a})  and Lemma  \ref{lem:ad} we obtain
\begin{equation}
\binom{n+p}{n} \ge \Pi_{k=0}^{r-1}\frac{ (n+1)^{r'-1}}{(k+1)^{r'-1}}\ge \Pi_{k=1}^r (2 \frac{s+k}{k}).\label{eq:ad6}
\end{equation}

Taking into account (\ref{super31b}), we obtain   (\ref{super31})  from (\ref{eq:ad6}).  This   proves (\ref{super1}). 

Since (\ref{super1})holds,   we  can apply Proposition \ref{pola} to get a $(s,r)$-elusive mapping $f \in Pol^p( \F ^n,\F ^m)$, whose monomial coefficients are algebraic numbers are $\exp (\frac{2\pi\sqrt{-1}}{p ^ i_j})$ or its real part.  

This completes the proof  of  Proposition \ref{super}.
\end{proof}
\medskip

Lemma \ref{lem:raz38}  and Proposition \ref{super}  yield immediately

\begin{corollary}\label{cor:circs} Assume that $r'$ grows  much slower  than $n$, e.g. $r' = const$  or $r' =  \ln \ln n$. Let
$p = (r' -1)(2r' -1)$.  Then there are sequences  of  polynomials $f_n\in Pol ^{p + r' +1}(\F^{3n})$,  whose coefficients  are algebraic numbers,  such that 
$$ L(f_n) \ge \Om( \frac{\lfloor{n \over r' (r' -1)} \rfloor ^{ {r' -3\over 2}}}{(r' ) ^ 4}).$$
\end{corollary}
\begin{proof}  Taking into account Lemma \ref{lem:raz38}  and Proposition \ref{super}, it suffices to   prove that  if $r' = \ln \ln n$  and  $n$ is sufficient    large, then    $(n +r -4)  ^ 4 \ge  r!$.
Clearly, $(n +r -4)  ^ 4 \ge  r!$   follows from  $r\ln r  <  \ln  n$.  Since  $ r   > \ln r $ for    sufficiently large  $r$, it suffices   to show that  $r ^2 <  \ln n $, or equivalently
$2 \ln r <  \ln \ln n$.    The  last inequality holds for   large $r$, since  $ 2 \ln r <  r  = \ln \ln n$.
\end{proof}

Note that  Corollary \ref{cor:circs}   yields  a much better  lower bound than  that in Proposition \ref{c45}, whose  assertion   we have compared  with a similar result  by Raz  and with 
that one by
Shoup  and Smolensky.  This demonstrates  the effectiveness of our methods.


\section*{Acknowledgements}   I am indebted  to  Pavel Pudlak  for his support, stimulating  helpful discussions and critical remarks. I am thankful to Gerhard Pfister for his explanation
of their results in \cite{G-P2007},  to Ran Raz for his motivating lecture in Prague \cite{Raz2009b}, and to  Sasha Sivatsky  for his helpful remarks  and to Partha Mukhopadhyay  for stimulating discussions. 
A part of this note has been written during my visit of MSRI, Berkeley, GIT, Atlanta, and    ASSMS, Government College University, Lahore-Pakistan. I thank  these institutions for their hospitality and  financial support.

\end{document}